\documentclass[11pt]{article}

\usepackage{lipsum}
\usepackage{graphicx}
\usepackage{float}
\usepackage{biblatex}
\usepackage[hidelinks]{hyperref}
\usepackage{xurl}
\hypersetup{breaklinks=true}
\usepackage{mathtools}
\usepackage{amssymb}
\usepackage{setspace}
\usepackage[utf8]{inputenc}
\usepackage[english]{babel}
\usepackage{amsmath}
\usepackage{tikz}
\usepackage{framed}
\usepackage[hang,flushmargin]{footmisc}
\hypersetup{colorlinks,linkcolor={blue},citecolor={green},urlcolor={red}}  
\usepackage{amsthm}
\usepackage{xcolor}

\usepackage{csquotes}

\newtheorem{theorem}{Theorem}

\newtheorem{corollary}{Corollary}

\newtheorem{definition}{Definition}

\newtheorem{lemma}{Lemma}

\newtheorem{remark}{Remark}

\makeatletter

\newcommand\ackname{Acknowledgements}
\if@titlepage
  \newenvironment{acknowledgements}{%
      \titlepage
      \null\vfil
      \@beginparpenalty\@lowpenalty
      \begin{center}%
        \bfseries \ackname
        \@endparpenalty\@M
      \end{center}}%
     {\par\vfil\null\endtitlepage}
\else
  
\fi

\title{Towards a monopole Fueter Floer homology I: \\a compactness theorem}
\author{Saman Habibi Esfahani}
\date{\today}

\begin{document}
\maketitle

\begin{abstract}
Motivated by a conjecture of Donaldson and Segal \cite{MR2893675}, we take a first step towards defining a new 3-manifold Floer theory, where the complex is defined by a count of Fueter sections of a hyperk\"ahler bundle over the 3-manifold with fibers modeled on the moduli space of centered monopoles on $\mathbb{R}^3$ with charge $k \in \mathbb{Z}_{\geq 2}$. The main difficulty in defining these counts comes from the non-compactness problems. In this writing, we prove a compactness theorem in this direction in the case of $k=2$. 
\end{abstract}

\tableofcontents

\section{Introduction}

The Fueter operator is a non-linear generalization of the Dirac operator defined over oriented Riemannian 3- and 4-manifolds, where the spinor bundle is replaced by a hyperk\"ahler bundle. We start with the 3-dimensional version as introduced by Taubes in \cite{MR1708781}. We address the 4-dimensional case in the appendix.

\begin{definition}
Let $(M,g)$ be an oriented Riemannian 3-manifold. Let $\pi: \mathfrak{X} \rightarrow M$ be a fiber bundle with fibers modeled on a hyperk\"ahler
manifold $(X,g_X,I,J,K)$
with an isometric bundle identification 
$I : STM \to \mathfrak{b}(\mathfrak{X})$,
where $STM$ is the unit tangent bundle of $M$ and  $\mathfrak{b}(\mathfrak{X})$ is the sphere bundle of the complex structures of the fibers of $\mathfrak{X}$. Let $\nabla$ be a connection on $\mathfrak{X}$. A section $f \in \Gamma (\mathfrak{X})$ is called a Fueter section if 
\begin{align*}
\mathfrak{F}(f) := I(e_1) \nabla_{e_1} f + I(e_2) \nabla_{e_2} f + I(e_3) \nabla_{e_3} f = 0 \in \Omega^0(M,f^*V\mathfrak{X}),     
\end{align*}
where $V\mathfrak{X} := ker(d\pi) : T \mathfrak{X} \rightarrow TM$ is the vertical bundle and $(e_1, e_2, e_3)$ is a local orthonormal frame on $M$. The operator $\mathfrak{F}$ is called the Fueter operator. This operator is well-defined and does not depend on the choices of local orthogonal frames.
\end{definition}

For any $k \in \mathbb{Z}_{\geq 2}$, we propose Casson-type counts and, more ambitiously, Floer homology groups of closed oriented Riemannian 3-manifolds, defined by a count of Fueter sections of a bundle with fibers modeled on the moduli spaces of centered monopoles on $\mathbb{R}^3$ with charge $k$. 

Let $(M,g)$ be an oriented Riemannian 3-manifold. Let $Fr_{SO(3)} \to M$ denote the $SO(3)$ frame bundle of $TM$. Let $\text{Mon}_k^{\circ}(\mathbb{R}^3)$ denote the moduli space of centered monopoles on $\mathbb{R}^3$ with charge $k \geq 2$. $\text{Mon}_k^{\circ}(\mathbb{R}^3)$ is a $(4k-4)$-dimensional non-compact hyperk\"ahler manifold \cite{MR934202}. There is an $SO(3)$ action on $\text{Mon}_k^{\circ}(\mathbb{R}^3)$ induced by the natural action of $SO(3)$ on $\mathbb{R}^3$.  We define the monopole bundle with charge $k$ by
\begin{align*}
    \mathfrak{X}_k:= Fr_{SO(3)} \times_{SO(3)} \text{Mon}_k^{\circ}(\mathbb{R}^3) \to M.
\end{align*}  

The Levi-Civita connection on the frame bundle, denoted by $\nabla_{LC}$, induces a connection on the associated monopole bundle $\mathfrak{X}_k$, and therefore, a choice of an auxiliary connection on $\mathfrak{X}_k$ is not needed. For any sectiont $f \in \Gamma(\mathfrak{X}_k)$, we have $\nabla f \in f^* \text{Vert}(\mathfrak{X}_k) \otimes T^*M$, where $\text{Vert}(\mathfrak{X}_k) = Fr_{SO(3)} \times_{SO(3)} T\mathfrak{X}_k$. At any point $x \in M$, let $f = [(e, \tilde{f})]$ for a section $e \in \Gamma(Fr_{SO(3)})$ such that at $x$ we have $\nabla_{LC} (d e) = 0$ and a map $\tilde{f}: M \to \text{Mon}_k^{\circ}(\mathbb{R}^3)$. We have 
\begin{align*}
    \nabla_X f(x) = [(e(x), d_x\tilde{f}(X))],
\end{align*}
for any $X \in T_xM$.

There is a natural identification $I: STM \to \mathfrak{b}(\mathfrak{X}_k)$ induced from the frame bundle structure. To see this, let $\text{Mon}_k^{\circ}(TM) \to M$ be the hyperk\"ahler bundle where the fiber above any $x \in M$ is defined to be the moduli space of centered monopoles on $T_xM$. There is a canonical bundle identification 
\begin{align*}
    \mathfrak{X}_k \cong 
    \text{Mon}_k^{\circ}(TM).
\end{align*}
Hence, any unit vector in $T_xM$ determines a unique complex structure on $\text{Mon}_k^{\circ}(T_xM)$, and therefore, on $(\mathfrak{X}_k)_x$.

Fueter operators over 3-manifolds are non-linear Fredholm operators with index zero. This motivates the following definition. 

\begin{definition}
Let $(M,g)$ be a closed oriented Riemannian 3-manifold. Let $\mathfrak{X}_k \to M$ be the monopole bundle with charge $k$. Let $\mathfrak{Fuet}_g(\mathfrak{X}_k)$ be the space of Fueter sections with respect to the Levi-Civita connection and the natural identification $I$. Let
\begin{align*}
  \mathfrak{Fuet}_g(\mathfrak{X}_k) =
  \{
  f \in \Gamma(\mathfrak{X}_k) \; | \; \mathfrak{F}(f) = 0
  \}.
\end{align*}
In the case $|\mathfrak{Fuet}_g(\mathfrak{X}_k)| < \infty$, we define the $k$-th monopole Fueter count of $(M,g)$ by
\begin{align*}
    m_k(M,g) := \# \mathfrak{Fuet}_g(\mathfrak{X}_k).
\end{align*}
This count can be understood as a $\mathbb{Z}_2$-count or, more ambitiously, one can hope to define a signed count,
\begin{align*}
    m_k(M,g) = \sum_{f \in \mathfrak{Fuet}_g(\mathfrak{X}_k)} \text{sign}(f).
\end{align*}
\end{definition}

These counts a priori depend on the Riemannian metric $g$. It would be optimistic to expect that these counts do not depend on the choice of the metric. One hopes to prove invariance for a generic choice of metric where a transversality condition holds or clarify how these counts change as the metric varies and potentially compensate for the changes to get invariants of 3-manifolds. 

More ambitiously, for each $k \in \mathbb{Z}_{\geq 2}$, one can move towards defining a Floer theory which categorifies the $k$-th monopole Fueter count, called the $k$-th monopole Futer Floer homology. The complex of this Floer homology is generated by the Fueter sections of $\mathfrak{X}_k \to M$, and the differentials are defined by a count of the 4-dimensional Fueter sections on $M \times \mathbb{R}$ of the pull-back bundle $\pi^* \mathfrak{X}_k \to M \times \mathbb{R}$ via the projection map $\pi: M \times \mathbb{R} \to M$, where the 4-dimensional Fueter sections interpolate between two 3-dimensional Fueter sections on $M$. A section $f \in \Gamma (\pi^* \mathfrak{X}_k)$ over $M \times \mathbb{R}$ is a 4-dimensional Fueter section if 
\begin{align*}
    \partial_t f = \sum_{i=1}^3 I(e_i) \nabla_{e_i} f, 
\end{align*}
where $t$ denotes the coordinate on the $\mathbb{R}$ component and $(e_1, e_2, e_3)$ is a local orthonormal frame on $M$. 

The main difficulty in defining these counts comes from the non-compactness problems and the potential wall-crossing phenomenon.

\hspace{5pt}

We give three items of motivation for this study.

\begin{itemize}
    \item Donaldson and Segal hinted at the idea of defining invariants of Calabi-Yau 3-folds by a count of monopoles on these manifolds \cite{MR2893675}. Moreover, they conjectured that these monopole counts should be equal to a weighted count of special Lagrangians in the Calabi-Yau 3-folds, where the weights are defined by a count of Fueter sections of monopole bundles over these special Lagrangians. The compactness problems are the major difficulty in proving the Donaldson-Segal conjecture. 
    
    Moreover, this observation of Donaldson and Segal suggests that Fueter sections are building blocks of the gluing construction of monopoles on Calabi-Yau 3-folds. The gluing argument is expected to be a higher-dimensional version of what appeared in the 3-dimensional case, as in \cite{MR614447, esfahani2022singular}.

    \item It is expected that the counts of Fueter sections of monopole bundles with charge 2 to be related to the Rozansky-Witten invariants of 3-manifolds \cite{MR1481135}. The potential wall-crossing phenomenon, i.e., the change in the count of Fueter sections as one varies the Riemannian metric, is a major difficulty in this direction. One hopes $k=2$ monopole Fueter Floer theory would categorify the Rozansky-Witten invariant of 3-manifolds. 

    \item Taubes suggested counts of Fueter sections of monopole bundles with $SO(3)$ actions, including the moduli spaces of centered monopoles. Moreover, he introduced the generalized Seiberg-Witten equations on 3-manifolds where one replaces the linear Dirac operator in the standard Seiberg-Witten equations with a Fueter operator on a non-linear hyperk\"ahler bundle. We hope some of the compactness analysis we develop here can be used in the study of Taubes' generalized Seiberg-Witten equations.
    
\end{itemize}

Following a theorem of Walpuski \cite{MR3718486}, a sequence of Fueter sections $(f_i)_{i=1}^{\infty}$ of a fiber bundle $\mathfrak{X}$ with fibers modeled on a compact hyperkähler manifold $X$ over a closed oriented Riemannian 3–manifold $(M,g)$ with bounded energy $(\frac{1}{2}\|\nabla f_i\|^2_{L^2})_{i=1}^{\infty}$, after passing to a subsequence, converges outside of a 1-dimensional closed rectifiable subset $S \subset M$ to a Fueter section $f$ in $C^{\infty}_{\text{loc}}$. The non-compactness along $S$ has two sources: 

\hspace{1pt}

(1) bubbling-off of holomorphic spheres in the fibers of $\mathfrak{X}$; 

\hspace{1pt}

(2) formation of non-removable singularities. 

\hspace{1pt}

When the fibers of the hyperk\"ahler bundle are non-compact, there is a third source of non-compactness: 

\hspace{1pt}

(3) A seqeunce of Fueter sections $(f_i)_{i=1}^{\infty}$ can diverge to infinity of $\mathfrak{X}$ on a subset of the 3-manifold $M$.

\hspace{1pt}

In this writing, we study the first two sources of non-compactness for Fueter sections of the monopole bundles with charge 2 and prove a compactness theorem in that direction. We will study the divergence to infinity in a sequential paper \cite{esfahani2023}. 

\begin{theorem}\label{Main-theorem}
Let $(M,g)$ be a closed oriented Riemannian 3-manifold. Let $\mathfrak{X}_2 \to M$ be the monopole bundle with charge 2. Let $K \subset \mathfrak{X}_2$ be a compact subset. Let $(f_i)_{i=1}^{\infty} \subset \Gamma(\mathfrak{X}_2)$ be a sequence of smooth Fueter sections such that $\text{Image}(f_i) \subset K$ for all $i$. Then, after passing to a subsequence, $(f_i)_{i=1}^{\infty}$ converges to a smooth Fueter section $f \in \Gamma (\mathfrak{X}_2)$ in $C^{\infty}(M)$.
\end{theorem}

\textbf{Plan of the paper.} In Section \ref{Preliminaries}, we recall Walpuski's compactness theorem and the background material on monopoles on $\mathbb{R}^3$. In Section \ref{sec:energy-bound}, we will prove an energy bound for Fueters sections of charge 2 monopole bundles. In Section \ref{sec:bub}, we will prove the vanishing of the bubbling locus for the Fueter sections of monopole bundles. In Section  \ref{sec:non-rem-sing}, we will show that the locus of non-removable singularities of Fueter sections of charge 2 monopole bundles vanishes. This completes the proof of Theorem \ref{Main-theorem}. In the appendix, we state and prove the 4-dimensional version of Theorem \ref{Main-theorem}. 

\vspace{7pt}

\noindent
\textbf{Acknowledgements.} This article is part of the Ph.D. thesis of the author at Stony Brook University \cite{MR4495257}. I am grateful to my advisor Simon Donaldson for suggesting this problem and also because of his guidance and support. Moreover, I would like to thank Lorenzo Foscolo, Thomas Walpuski, and Yao Xiao for helpful conversations. 

\section{Preliminaries}\label{Preliminaries}

In this section, we review the necessary background, which consists of a review of Walpuski's compactness theorem for Fueter sections and the geometry of the moduli spaces of centered monopoles on $\mathbb{R}^3$.

\subsection{Walpuski's compactness theorem} We start by recalling Walpuski's compactness theorem. The theorem is originally stated in the case where the hyperk\"ahler bundle $\mathfrak{X}$ is compact; however, the theorem and its proof, essentially without any modification, generalize to the case where $\mathfrak{X}$ is not necessarily compact, but $\text{Image} (f_i)$ lies inside a compact subset $K \subset \mathfrak{X}$ for all $i \in \{1, 2, \hdots \}$.

\begin{theorem}[Walpuski \cite{MR3718486}]\label{prop:wal3}
Let $(M,g)$ be a closed oriented Riemannian 3-manifold. Let $\mathfrak{X} \to M$ be a hyperk\"ahler bundle with fibers modeled on a hyperk\"ahler manifold $(X,g_X,I,J,K)$. Suppose the bundle is equipped with a connection $\nabla$ and an isometric bundle identification $I: STM \to \mathfrak{b(X)}$. Let $K \subset \mathfrak{X}$ be a compact subset. Let $(f_i)_{i=1}^{\infty}$ be a sequence of Fueter sections on this bundle such that $\text{Image}(f_i) \subset K$, and
\begin{align*}
    \mathcal{E}(f_i) := \frac{1}{2}
\int_M | \nabla f_i |^2 \leq c_e,
\end{align*}
for a constant $c_e$. 

Then, after passing to a subsequence, the following holds:

\begin{itemize}
\item There exists a closed rectifiable subset $S \subset M$, called the singular locus, with $\mathcal{H}^1(S) < \infty$ and a Fueter section $f \in \Omega^0 (M \setminus S, \mathfrak{X} )$ such that $f_i |_{M \setminus S}$ converges to $f$ in $C^{\infty}_{loc}(M\setminus S)$.

\item There exist a constant $\epsilon_0 > 0$ and an upper semi-continuous function $\Theta : S \rightarrow [\epsilon_0, \infty)$ such that the sequence of measures $\mu_i := |\nabla f_i|^2 \mathcal{H}^3$ converges weakly to $\mu := |\nabla f|^2 \mathcal{H}^3 + \Theta \mathcal{H}^1|_S$. The singular locus $S$ decomposes as 
\begin{align*}
   S = \Gamma \cup \text{sing} (f), 
\end{align*}
where $\Gamma$, called the bubbling locus, is given by
\begin{align*}
   \Gamma := \text{supp} (\Theta \mathcal{H}^1_{|_S}),
\end{align*}
and $\text{sing} (f)$, called the locus of non-removable singularities of $f$, is defined by
\begin{align*}
    \text{sing} (f) := \{ x \in M \; | \; 
\limsup_{r \downarrow 0} \frac{1}{r} \int_{B_r (x)} |\nabla f|^2 > 0 \}.
\end{align*}
$\Gamma$ is $\mathcal{H}^1$-rectifiable and 
$\mathcal{H}^1 (\text{sing} (f)) = 0.$

\item For each smooth point $x \in \Gamma$, i.e., a point $x$ where $T_x \Gamma$ exists and $x \notin \text{sing}(f)$, there exists a non-trivial $-I(v)$-holomorphic sphere 
\begin{align*}
    \zeta_x : \nu_x \Gamma \cup \{\infty\} \to \mathfrak{X}_x,
\end{align*}
where $v$ is a unit tangent vector in $T_x \Gamma$.
\end{itemize}
\end{theorem}

The analysis of Fueter sections relies on an almost monotonicity formula. 

\begin{theorem}[Walpuski \cite{MR3718486}]\label{almost-monotonicity}
    Let $f \in \Gamma(\mathfrak{X})$ be a Fueter section. For all $x \in M$ and $0<s<r\leq r_0(x)$, where $r_0(x)$ is the injectivity radius of $(M,g)$ at $x$, we have
    \begin{align}\label{almost-monotonicity1}
        \frac{e^{cr}}{r} \int_{B_r(x)}|\nabla f|^2 - \frac{e^{cs}}{s} \int_{B_s(x)}|\nabla f|^2 \geq 2\int_{B_r(x) \setminus B_s(x)}\frac{1}{\rho}|\nabla_r f|^2 - c(r^2 - s^2),
    \end{align}
for a constant $c$ and $\rho = d(x,\cdot)$.

Let $f: \mathbb{R}^3 \to X$ be a Fueter map, i.e., at each $x \in \mathbb{R}^3$ we have
\begin{align*}
    I d_x f(e_1) + J d_x f(e_2) + K d_x f(e_3) = 0. 
\end{align*}
Then, we have the following equality,
\begin{align}\label{almost-monotonicity2}
        \frac{1}{r} \int_{B_r(x)}|d f|^2 - \frac{1}{s} \int_{B_s(x)}|d f|^2 = 2\int_{B_r(x) \setminus B_s(x)}\frac{1}{\rho}|\partial_r f|^2.
\end{align}
\end{theorem}

\begin{remark}
    Formula \ref{almost-monotonicity1} in \cite{MR3718486} is stated without the constant 2 in front of the integral $\int_{B_r(x) \setminus B_s(x)}\frac{1}{\rho}|\nabla_r f|^2$. This is because there the formula is stated for more general sections where they satisfy a perturbed Fueter equation. From Walpuski's proof, one can see that when $f$ is a Fueter section, one has this more optimal inequality. However, the constant in the formula does not affect our analysis.
\end{remark}

\subsection{Moduli spaces of centered monopoles on $\mathbb{R}^3$}

This subsection briefly reviews the basic results about the geometry of the moduli spaces of centered monopoles on $\mathbb{R}^3$. For a more detailed account, consult the book by Atiyah and Hitchin \cite{MR934202}.

Let $P \rightarrow \mathbb{R}^3$ be a principal $SU(2)$-bundle. A pair $(A, \varPhi)$ of a connection $A$ on $P$ and a section $\varPhi$ of the associated adjoint bundle $\mathfrak{g}_P$ is called a monopole if it satisfies the Bogomolny equation,
\begin{align*}
    *F_A = d_A \varPhi.
\end{align*}
Let $(A, \varPhi)$ be a monopole with $\|F_A\|_{L^2} < \infty$. The mass is defined to be $m = \lim_{|x| \to \infty} |\varPhi(x)|$, which can be normalized to be 1. Moreover, the charge, denoted by $k$, is defined by 
\begin{align*}
    k = \text{deg} 
    \left ( \frac{\varPhi}{|\varPhi|} : S^2_R(0) \to S^2_1(0) \subset \mathfrak{su}(2) \right) \in \mathbb{Z}_{\geq 0},
\end{align*}
for a sufficiently large $R$.

The moduli space of monopoles on $\mathbb{R}^3$ with mass 1 and charge $k$ is defined by
\begin{align*}
    \text{Mon}_k(\mathbb{R}^3) = 
    \{
    (A, \varPhi) \; | \; *F_A = d_A \varPhi, \; F_A \in L^2, \;
    m = 1, \;
    \text{charge}(A, \varPhi) = k
    \}
    /
    \mathcal{G},
\end{align*}
where $\mathcal{G}$ denotes the gauge group.

The translations on $\mathbb{R}^3$ act naturally on the moduli spaces of monopoles on $\mathbb{R}^3$. The moduli space of centered $k$-monopoles on $\mathbb{R}^3$ is defined by 
\begin{align*}
    \text{Mon}_k^{\circ}(\mathbb{R}^3) = \frac{\text{Mon}_k(\mathbb{R}^3)}{\mathbb{R}^3}.
\end{align*}

\begin{lemma}[Atiyah-Hitchin \cite{MR934202}]
The moduli space of centered monopoles on $\mathbb{R}^3$ with charge $k \in \mathbb{Z}_{\geq 2}$ is a $(4k-4)$-dimensional hyperk\"ahler manifold. 
\end{lemma}

For any axis in $\mathbb{R}^3$, i.e., a line which passes through the origin, and any charge $k \in \mathbb{Z}_{\geq 2}$, there is a unique centered monopole in $\text{Mon}_k^{\circ}(\mathbb{R}^3)$ which is invariant under rotation around this given axis. This introduces an $\mathbb{RP}^2$ family of monopoles in $\text{Mon}_k^{\circ}(\mathbb{R}^3)$, called the axisymmetric monopoles and denoted by $\Sigma_k$. This surface is a minimal $\mathbb{RP}^2$ in $\text{Mon}_k^{\circ}(\mathbb{R}^3)$. 

In this writing, we mainly focus on the case where $k=2$. The 4-manifold $\text{Mon}_2^{\circ}(\mathbb{R}^3)$ is called the Atiyah-Hitchin space. It is an ALF space and is asymptotic to a $U(1)$-bundle over $\mathbb{R}^3/\mathbb{Z}_2$. There is a 2-convex function
\begin{align*}
r^2: \text{Mon}_2^{\circ}(\mathbb{R}^3) \to \mathbb{R},
\end{align*} 
where $r$ denotes the geodesic distance from the minimal surface $\Sigma_2$. This can be used to prove that there is no other closed minimal surface in the Atiyah-Hitchin space.

\begin{theorem}[Tsai-Wang \cite{MR4516042}] \label{Tsai-Wang-thm}
Let $\Sigma_2$ denote the minimal surface of axisymmetric monopoles in the Atiyah-Hitchin space. Then,
\begin{itemize}
    \item $\Sigma_2$ is the only compact minimal surface in $\text{Mon}_2^{\circ}(\mathbb{R}^3)$.
    \item $\Sigma_2$ is a calibrated submanifold, and therefore, it minimizes the area within its homology class.
    \item There is no compact 3-dimensional minimal submanifold in $\text{Mon}_2^{\circ}(\mathbb{R}^3)$.
\end{itemize}
\end{theorem}

 \section{The energy bound}\label{sec:energy-bound}

An essential assumption in Walpuski's compactness theorem is that there is a bound on the energy of the Fueter sections of the sequence, $1/2\|\nabla f_i\|_{L^2}^2 \leq c_e$. This assumption is required since, unlike the theory of pseudo-holomorphic curves in symplectic manifolds with compatible almost complex structures, the energy of Fueter sections is not necessarily a topological quantity. 

The following theorem asserts that in the case of the charge 2 monopole bundles, where the images of the Fueter sections fall inside a compact subset of the bundle, there is a bound on the energy of the Fueter sections.

\begin{theorem}\label{energy-bound}
    Let $(M,g)$ be a closed oriented Riemannian 3-manifold. Let $\mathfrak{X}_2 \to M$ be the monopole bundle with charge 2. Let $K \subset \mathfrak{X}_2$ be a compact subset of the monopole bundle. There is a constant $c_K$ which depends on $K$ and $(M,g)$ such that for any Fueter section $f \in \Gamma(\mathfrak{X}_2)$ with $\text{Image}(f) \subset K$, we have 
    \begin{align*}
        \|\nabla f\|_{L^2(M)} \leq c_K r(K),
    \end{align*}
    here 
    \begin{align*}
        r(K) = \sup_{y \in K} \{ r(y) \} < \infty,
    \end{align*}
    where $r: \text{Mon}_2^{\circ}(\mathbb{R}^3) \to \mathbb{R}$ is the radius function.
\end{theorem}

\begin{proof}
Let $x \in M$. Let $\delta^1, \delta^2$, and $\delta^3$ be a local orthonormal coframe on a neighbourhood of $x$, dual to a local frame $e_1, e_2$, and $e_3$, respectively. Let $\Lambda \in \Omega^3(M)$ be a 3-form on $M$, which in this local coordinate system is given by
\begin{align*}
    \Lambda = \sum_{i=1}^3 \delta^i \wedge f^*(\omega_{I(e_i)}),
\end{align*}
where $\omega_{I(e_i)}$ is the symplectic structure on the fibers of $\mathfrak{X}_2$ associated to the complex structure $I(e_i)$. This 3-form is independent of the chosen local orthonormal frame, and therefore, it is defined globally on $M$. 

For any section $f \in \Gamma(\mathfrak{X}_2)$, we have 
\begin{align*}
    \|\nabla f\|_{L^2}^2 = \| \mathfrak{F}(f)\|_{L^2}^2 - 2\int_{M} \Lambda.
\end{align*}
This identity follows from a pointwise computation as in \cite[Lemma 2.2]{MR2529942}. Therefore, if $f$ is a Fueter section,
\begin{align*}
    \mathcal{E}(f) =  - \int_{M} \Lambda.
\end{align*}
The standard $SO(3)$ action on $\mathbb{R}^3$ induces an action of $SO(3)$ on $\text{Mon}^{\circ}_k(\mathbb{R}^3)$ for any $k \in \mathbb{Z}_{\geq 2}$. Hence, this action induces an action of $SO(3)$ on the 2-sphere family of K\"ahler structures of $\text{Mon}^{\circ}_k(\mathbb{R}^3)$, denoted by $\mathfrak{b}(\text{Mon}^{\circ}_k(\mathbb{R}^3))$,
\begin{align*}
    \mathfrak{b}(\text{Mon}^{\circ}_k(\mathbb{R}^3)) = \{aI + bJ + cK \; | \; a^2+b^2+c^2\} \cong S^2.
\end{align*}
This induced action is the standard action of $SO(3)$ on $S^2 \subset \mathbb{R}^3$. Such an $SO(3)$ action on a hyperk\"ahler manifold is called a permuting action. 

In the presence of a permuting $SO(3)$-action on any hyperk\"ahler manifold $X$, every symplectic form $\omega$ in the 2-sphere family of K\"ahler structures is exact. Let $(\tau_1, \tau_2, \tau_3)$ be a basis for $\mathfrak{so}(3)$. Then,
\begin{align*}
    \omega_i = d \alpha_i \quad \text{ with } \quad \alpha_i := \iota_{v_{i+1}}\omega_{i+2} \in \Omega^1(M),
\end{align*}
where $v_i$ is the generating vector field of the infinitesimal action of $\tau_i$ on $X$, $i \in \{1,2,3\}$, and with the convention $3+1=1$.

In the fiber bundle setting, we have $\omega_{I(e_i)} = d \alpha_{I(e_i)}$, where $\omega_{I(e_i)}$ and $\alpha_{I(e_i)}$ are forms on the fibers of the hyperk\"ahler bundle. 

Let $x \in M$. Let $B_{\epsilon}(x)$ be a coordinate neighbourhood around $x$. By integration by parts, 
\begin{align*}
    \frac{1}{2}\|\nabla f\|_{L^2(B_{\epsilon}(x))}^2 &=  - \int_{B_{\epsilon}(x)} \sum_{i=1}^3 \delta^i \wedge f^*(d\alpha_{I(e_i)}) = 
    - \int_{B_{\epsilon}(x)} \sum_{i=1}^3 \delta^i \wedge d(f^*(\alpha_{I(e_i)}))\\& =
     - \int_{B_{\epsilon}(x)} \sum_{i=1}^3 d\delta^i \wedge f^*(\alpha_{I(e_i)}) +  \int_{\partial B_{\epsilon}(x)} \sum_{i=1}^3 \delta ^i \wedge f^* (\alpha_{I(e_i)}).
\end{align*}
Let $\beta \in \Omega^2(M)$ be the 2-form which on each coordinate ball is defined by 
\begin{align*}
    \beta = \sum_{i=1}^3 \delta ^i \wedge f^* (\alpha_{I(e_i)}).
\end{align*}
Note that this 2-form is independent of the chosen orthonormal coframe, and therefore, it is defined globally on $M$. 

Let $B_1, \hdots, B_n$ be $n$ disjoint coordinate balls such that $M = \overline{B_1} \cup \hdots \cup \overline{B_n}$. The integral of $\beta$ over the boundaries $\partial B_1, \hdots, \partial B_n$ cancel each other since the same integrand at each point in $\partial B_1 \cup \hdots \cup \partial B_n$ will appear in two different integrals with opposite signs. 

Let $(\delta^{1,j}, \delta^{2,j}, \delta^{3,j})$ be an orthonormal coframe on $B_j$ dual to $(e_{1,j}, e_{2,j}, e_{3,j})$ for each $j \in \{1, \hdots, n\}$. By the H\"older's inequality, we have
\begin{align*}
    \frac{1}{2}\|\nabla f\|_{L^2(M)}^2 &= 
     - \sum_{j=1}^n \int_{B_j} \sum_{i=1}^3 d\delta^{i,j} \wedge f^*(\alpha_{I(e_{i,j})})  \\& \leq  \sum_{j=1}^n \sum_{i=1}^3
     \|d \delta^{i,j}\|_{L^2(B_j)} \|f^*(\alpha_{I(e_{i,j})})\|_{L^2(B_j)}.
\end{align*}
Terms $\|d \delta^{i,j}\|_{L^2}$ depend only on the geometry of $(M,g)$, and therefore, $\|d \delta^{i,j}\|_{L^2} \leq c_1$ for a constant $c_1$. 

On the other hand, at any point $x \in M$, we have the pointwise estimate
\begin{align*}
    |f_x^*(\alpha_{I(e_i)})| = 
    |\alpha_{I(e_i)}(f(x)) \circ d_xf| \leq \left( \sup_{y \in M} |\alpha_{I(e_i)}(f(y))| \right) |d_xf|.
\end{align*}
Let $r: \text{Mon}_2^{\circ}(\mathbb{R}^3) \to \mathbb{R}$ denote the geodesic distance from the axisymmetric minimal $\mathbb{RP}^2$ in the moduli spaces of centered 2-monopoles. We have 
\begin{align*}
    |\alpha_{I(e_i)}(f(x))| \leq c_2 r(f(x)),
\end{align*}
for a constant $c_2$.

To see this, note that the monopole bundle with charge 2 over $M$ is asymptotic to a $U(1)$-bundle over $T^*M/\mathbb{Z}_2$. 1-form $\alpha_{I(e_i)}$, in the Gibbons-Hawking coordinates, is asymptotic to $x_i \wedge \theta_0 + dx_{i+1} \wedge dx_{i+2}$ which grows linearly.

Therefore, when $\text{Image}(f) \subset K$, we have
\begin{align*}
    |f^*(\alpha_{I(e_i)})| \leq c_2 r(K) |df(e_i)|.
\end{align*}
Hence,
\begin{align*}
    \|f^*(\alpha_{I(e_i)})\|_{L^2(M)} \leq c_3 r(K) \|df\|_{L^2(M)} 
    \leq c_4 r(K) \|\nabla f\|_{L^2(M)}.
\end{align*}
for constants $c_3$ and $c_4$.

This shows
\begin{align*}
    \frac{1}{2}\|\nabla f\|_{L^2(M)}^2 =  - \int_M \Lambda \leq c_K r(K) \|\nabla f\|_{L^2(M)},
\end{align*}
for a constant $c_K$, and therefore, 
\begin{align*}
    \frac{1}{2}\|\nabla f\|_{L^2(M)} \leq c_K= r(K).
\end{align*}
\end{proof}
    
\section{Bubbling}\label{sec:bub} 

In this section, we show that in the case where the fibers of the hyperk\"ahler bundle are modeled on the moduli spaces of centered monopoles on $\mathbb{R}^3$, the bubbling phenomenon does not occur.

\begin{theorem}\label{no-bubbling-locus}
    Let $(M,g)$ be a closed oriented Riemannian 3-manifold. Let $\mathfrak{X}_k \to M$ be the monopole bundle with charge $k$, and let $K \subset \mathfrak{X}_k$ be a compact subset. Let $(f_i)_{i=1}^{\infty}$ be a sequence of Fueter sections such that $\text{Image}(f_i) \subset K$. Then, after passing to a subseqeunce, $f_i \to f$ in $C^{\infty}_{\text{loc}}$, away from a closed rectifiable set $S \subset M$ where $S = \text{sing}(f)$, and therefore, $\mathcal{H}^1(S) = 0$. In other words, the bubbling locus $\Gamma \subset S$ vanishes.
\end{theorem}

The rest of this section is devoted to the proof of Theorem \ref{no-bubbling-locus}. The proof is based on an observation of Walpuski that at any smooth point in the bubbling locus $x \in \Gamma$, the “lost energy” goes into the formation of bubbles in the fiber above $x$.

\begin{theorem}[Walpuski \cite{MR3718486}]\label{bubble} Suppose we are in the situation of Theorem \ref{prop:wal3}. Suppose the bubbling locus is smooth at $x \in \Gamma$, i.e., $T_x \Gamma$ exists and $x \notin \text{sing}(f)$. Then, there exists a non-trivial $(-I(v))$-holomorphic sphere $\zeta_x : \nu_x \Gamma \cup \{\infty\} \to \mathfrak{X}_x$, where $v$ is a unit tangent vector in $T_x \Gamma$, and  $\nu_x \Gamma$ is the fiber of the normal bundle of $\Gamma \subset M$  at $x$. 
\end{theorem}

Walpuski's theorem motivates the study of holomorphic spheres in the fibers of $\mathfrak{X}_k$.

\begin{lemma}\label{no-bubble}
The moduli spaces of centered monopoles on $\mathbb{R}^3$ with charge $k \in \mathbb{Z}_{\geq 2}$ do not contain any non-constant holomorphic sphere with respect to any of the complex structures in the 2-sphere family of complex structures on $\text{Mon}^{\circ}_k(\mathbb{R}^3)$.
\end{lemma}

\begin{proof}
As mentioned earlier, all the symplectic forms in the 2-sphere family of symplectic structures on the moduli spaces of centered monopoles on $\mathbb{R}^3$ are exact.

There is no $J$-holomorphic sphere in any exact symplectic manifold $(X,\omega)$ for any $J$ compatible with $\omega$. The reason is that if $f : S^2 \rightarrow X$ is a non-trivial $J$-holomorphic sphere, 
\begin{align*}
\int_{S^2} f^* \omega = \int_{S^2} \omega(\frac{\partial f}{\partial x_1},\frac{\partial f}{\partial x_2}) dx_1 dx_2 &= \int_{S^2} \omega(\frac{\partial f}{\partial x_2},J\frac{\partial f}{\partial x_2}) dx_1 dx_2 
\\&=\int_{S^2} \|\frac{\partial f}{\partial x_2}\|^2 dx_1 dx_2 > 0,
\end{align*}
for local orthonormal frames $(\partial x_1, \partial x_2)$ at each point $x \in S^2$.

But on the other hand since $\omega$ is exact $\omega = d \alpha$ and by the Stokes' theorem we have
\begin{align*}
\int_{S^2} f^* \omega&= \int_{S^2} df^* \alpha = 0.
\end{align*}
\end{proof}

\begin{proof}[Proof of Theorem \ref{no-bubbling-locus}] Lemma \ref{no-bubble} and Theorem \ref{bubble} show that the smooth part of the bubbling locus vanishes. 

From Theorem \ref{prop:wal3} we have $\mathcal{H}^1(\Gamma) < \infty$, and therefore, the Hausdorff dimension of $\Gamma$ is at most one. However, since in the case of monopole bundles $\Gamma$ does not contain any smooth point, the Hausdorff dimension of $\Gamma$ is strictly less than one. Moreover, since $\mathcal{H}^1(\text{sing}(f)) = 0$, the Hausdorff dimension of $S$ is strictly less that one too. This shows that
    \begin{align*}
        \Theta \mathcal{H}^1_{|_{S}}=0,
    \end{align*}
and therefore,
    \begin{align*}
        \Gamma = \text{supp}(\Theta \mathcal{H}^1_{|_{S}}) = 
        \emptyset.
    \end{align*}
\end{proof}

\section{Non-removable singularities}\label{sec:non-rem-sing}

In this section, we study the non-removable singularities of Fueter sections. In Subsection \ref{non-removable-tri-hol}, we will prove the existence of tangent maps for Fueter sections and see that a tangent map of a Fueter section at a non-removable singularity satisfies the tri-holomorphic equation. In Subsection \ref{analysis-tri-holomorphic}, we will see that the image of a tri-holomorphic map is a minimal surface. Using this property, in Subsection \ref{tri-hol-mon}, we will rule out the formation of non-removable singularities for a sequence of Fueter sections of monopole bundles with charge 2.

\subsection{Non-removable singularities and tri-holomorphic-maps}\label{non-removable-tri-hol}

We start with the definition of tri-holomorphic maps.

\begin{definition}
Let $(X, g_X, I, J, K)$ be a hyperk\"ahler manifold. Let $\mathfrak{b}(X)$ be the 2-sphere family of complex structures on $X$. Let $I: S^2 \to \mathfrak{b}(X)$ be an isometric identification. A map $f: S^2 \to X$ is called an $I$-tri-holomorphic map if at each point $x \in S^2$,
\begin{align*}
        \frac{\partial f}{\partial x_1}(x)  + I(x) \frac{\partial f}{\partial x_2}(x)  = 0,
\end{align*}
where $\partial x_1, \partial x_2$ is a local orthogonal frame at $x \in S^2$.

In other words, at each point $x \in S^2$, the map $d_xf$ is complex linear with respect to the complex structure $I(x)$ on $X$. 
\end{definition}

We first state the main theorems of this section, Theorems \ref{sing-fueter}, \ref{tri-hol-image}, and \ref{non-constant-tangent}. The rest of this section is dedicated to the proofs of these theorems.

\begin{theorem}\label{sing-fueter}
Suppose we are in the situation of Theorem \ref{prop:wal3} with $\Gamma = \emptyset$. Let $x \in \text{sing}(f)$. Then, $f$ has a tangent map at $x$: let $(r_i)_{i=1}^{\infty}$ be a sequence of positive numbers converging to 0. Let $f_{r_i}$ be the map defined by 
\begin{align*}
f_{r_i}(y) = f \circ \exp_x \circ \lambda_{r_i} (y), 
\end{align*}
on an open subset of $T_xM$, where $\lambda_{r_i} (y) = r_i y$. Then, there exists a subsequence of sections $f_{r_i}$, which converges weakly to a map
\begin{align*}
\varPhi_x : T_x M \setminus \{0\} \to \mathfrak{X}_x,    
\end{align*}
such that 
\begin{align*}
        \varPhi_x(y) = \varPhi_x(ry), \quad \quad \forall r > 0.
\end{align*}
\end{theorem}

By restricting the resulting tangent map $\varPhi_x$ of Theorem \ref{sing-fueter} to the unit sphere $S^2 \subset T_xM \cong \mathbb{R}^3$, we get a map 
\begin{align*}
    \varphi_x: S^2 \to \mathfrak{X}_x.
\end{align*}
This map is essential in the study of the singular set. 

\begin{theorem}\label{tri-hol-image}
    Suppose we are in the situation of Theorem \ref{prop:wal3} with $\Gamma = \emptyset$. The map  $\varphi_x: S^2 \to \mathfrak{X}_x$ is a $-I$-tri-holomorphic map. Moreover, the image of $\varphi_x$ in $\mathfrak{X}_x$ is  contractible.
\end{theorem}

Note that although the image of the map $\varphi_x$ of Theorem \ref{tri-hol-image} is contractible, it is not trivial. 

\begin{theorem}\label{non-constant-tangent}
    Suppose we are in the situation of Theorem \ref{prop:wal3} with $\Gamma = \emptyset$. Let $x \in \text{sing}(f)$. Any tangent map $\varPhi_x$ is non-trivial. In particular, the image of $\varphi_x$ in $\mathfrak{X}_x$ is non-trivial.
\end{theorem}

We start with some necessary lemmas for the proofs of Theorems \ref{sing-fueter}, \ref{tri-hol-image}, and \ref{non-constant-tangent}.

Let the density function $\theta_{\varPhi}$ of a Fueter section $\varPhi$ be the function defined by
\begin{align}\label{density}
    \theta_{\varPhi}(x) = \lim_{\rho \downarrow 0} \frac{1}{\rho} \int_{B_{\rho}(x)} |\nabla \varPhi|^2,
\end{align}
where $x$ can be a smooth or a singular point of $\varPhi$.

\begin{lemma}
     Let $(X, g_X, I, J, K)$ be a hyperk\"ahler manifold. Let $S \subset \mathbb{R}^3$ be a closed subset. Let $\varPhi : \mathbb{R}^3 \setminus S \to X$ be a Fueter map with $S = \text{sing}(\varPhi)$ and $\mathcal{H}^1(S) = 0$. Moreover, suppose the almost monotonicity formula \ref{almost-monotonicity1} holds for all $x \in \mathbb{R}^3$ with $c=0$. Then, the limit \ref{density} exists at all points, and therefore, the density function $\theta_{\varPhi}$ is defined for all $x \in \mathbb{R}^3$.
\end{lemma}

\begin{proof}

If $x \in M \setminus S$, then $\varPhi$ is smooth on a neighborhood of $x$, and therefore, Limit \ref{density} is 0 at $x$. 

Suppose $x \in S$. The almost monotonicity formula at $x$ reads as
\begin{align*}
            \frac{1}{r} \int_{B_r(x)}|d \varPhi|^2 - \frac{1}{s} \int_{B_s(x)}|d \varPhi|^2 \geq 2 \int_{B_r(x) \setminus B_s(x)}\frac{1}{\rho}|\partial_r \varPhi|^2,
\end{align*}
for $0 < s < r < r_0(x)$.

Note that the right-hand side of the almost monotonicity at $x$ is non-negative, and therefore, the function defined by
\begin{align*}
    g_{\varPhi}(\rho) := \frac{1}{r} \int_{B_{\rho}(x)}|d \varPhi|^2,
\end{align*}
is a non-negative, non-decreasing function in $\rho$; hence, Limit \ref{density} exists. 
\end{proof}

In the following lemma, we show that the almost monotonicity formula holds for the limiting Fueter section $f$.

\begin{lemma}\label{non-smooth-almost}
Suppose we are in the situation of Theorem \ref{prop:wal3} with $\Gamma = \emptyset$. Let $f$ be the limiting Fueter section. Then, the almost monotonicity formula \ref{almost-monotonicity1} holds for $f$ at any point $x \in M$. 
\end{lemma}

\begin{proof}
If $x \in M \setminus S$, then $f$ is smooth on a neighborhood of $x$, and therefore, Walpuski's argument for the almost monotonicity formula holds. 

Suppose $x \in S = \text{sing}(f)$. Following Theorem \ref{almost-monotonicity}, the almost  monotonicity formula holds for the smooth Fueter sections $f_i$ at $x$. Since $\Gamma = \emptyset$, there is no energy loss, and therefore, on any neighborhood of $x$, we have $\|\nabla f_i \|_{L^2} \to \|\nabla f \|_{L^2}$ as $i \to \infty$. Therefore, by taking limit $i \to \infty$ of the almost  monotonicity formula for $f_i$ at $x$, we get the formula for $f$.
\end{proof}

\begin{proof} [Proof of Theorem \ref{sing-fueter}] Let $N$ be sufficiently large such that $r_i < r_0(x)$ for all $i \geq N$. Let $S_i$ be the singular set of $f_{r_i}$,
\begin{align*}
    S_i = \lambda_{r_i}^{-1} \circ \exp_x^{-1}(\text{sing}(f) \cap B_{r_0(x)}(x)).
\end{align*}
Let $S = \cup_{i=N}^{\infty} S_i$. We have $\mathcal{H}^1(S) = 0$. Each $f_{r_i}$ is a smooth Fueter section on the open set $B_1(0) \setminus S \subset \mathbb{R}^3$. Furthermore,
\begin{align*}
    \frac{1}{2}\|\nabla f_{r_i}\|_{L^2(B_1(0))} \leq 
    \frac{1}{2}\|\nabla f\|_{L^2(B_{r_0}(x))} \leq c_e,
\end{align*} 
and therefore, the energy of sections $f_{r_i}$ is bounded on this region. Moreover, $\text{Image}(f_{r_i}) \subset \text{Image}(f) \subset K$. 

Although Fueter sections $(f_{r_i})_{i=1}^{\infty}$ are not necessarily smooth, Walpuski's compactness argument holds for them: for any $y \in B_1(0) \setminus S$, there exists $0<r<1$ such that for all $i$ sufficiently large
\begin{align*}
    \frac{1}{r} \int_{B_r(y)} | \nabla f_{r_i} |^2 \leq \epsilon_0.
\end{align*}
By the $\epsilon$-regularity \cite[Proposition 3.1.]{MR3718486}, there is a uniform bound on $|\nabla f_{r_i}|$ on a smaller open set $B_{r/4}(y)$. It follows from Arzel\'a–Ascoli that there is a convergent subseqeunce of $f_{r_{i_j}} \to \varPhi_x$ on $B_1(0) \setminus S$ in $C^{\infty}_{\text{loc}}$. 

For any sufficiently small $\sigma > 0$, 
\begin{align*}
\frac{1}{\sigma} \int_{B_{\sigma}(0)}|\nabla f_{r_i}|^2  = 
\frac{1}{\sigma r_i} \int_{B_{\sigma r_i}(x)}|\nabla f|^2.
\end{align*}
Let $i \to \infty$. Since there is no loss of energy, we have
\begin{align*}
\frac{1}{\sigma} \int_{B_{\sigma}(0)}|\nabla f_{r_i}|^2 \to  
\frac{1}{\sigma} \int_{B_{\sigma}(0)}|d \varPhi_x|^2.
\end{align*}
Therefore, as $i \to \infty$, the following limit exists,
\begin{align*}
\frac{1}{\sigma r_i} \int_{B_{\sigma r_i}(x)}|\nabla f|^2 \to \theta_f(x),
\end{align*}
and
\begin{align*}
\frac{1}{\sigma} \int_{B_{\sigma}(0)}|d\varPhi_x|^2  = 
\theta_f(x).
\end{align*}
Hence, the function 
\begin{align*}
    g_{\varPhi_x}(\sigma) := \frac{1}{\sigma} \int_{B_{\sigma}(0)}|d\varPhi_x|^2,
\end{align*}
is constant and independent of $\sigma$. In fact, 
\begin{align*}
    \frac{1}{\sigma} \int_{B_{\sigma}(0)}|d\varPhi_x|^2 = \theta_{\varPhi_x}(0) = \theta_f(x),
\end{align*}
for all $\sigma > 0$.

Lemma \ref{non-smooth-almost} shows that the almost monotonicity formula holds for $f$ with some constant $c$. By scaling, we see that the almost monotonicity formula holds for $f_{r_i}$. By Restricting to radii $0 < s < r < 1$, the constant $c \to 0$ as $i \to \infty$. Since there is no loss of energy, by taking the limit of the almost monotonicity for $f_{r_i}$, we get the almost monotonicity formula for $\varPhi_x$, with constant $c=0$. 
\begin{align*}
    \frac{1}{r} \int_{B_r(0)}|d \varPhi_x|^2 - \frac{1}{s} \int_{B_s(0)}|d \varPhi_x|^2 = 0 \geq 2 \int_{B_r(0) \setminus B_s(0)}\frac{1}{\rho}|\partial_r \varPhi_x|^2,
\end{align*}
hence, 
\begin{align*}
    \partial_r \varPhi_x = 0.
\end{align*}
Therefore, the tangent map of $f$ at $x$ is homogenous,
\begin{align*}
    \varPhi_x(y) = \varPhi_x(ry), \quad \quad \forall r > 0,
\end{align*}
and therefore, it can be extended to the whole ray $\mathbb{R}^+ \cdot y$ if it is defined at a point $y$.

As before, since $\text{Mon}_2^{\circ}(\mathbb{R}^3)$ does not contain any holomorphic sphere, the singular locus $S_{\varPhi_x}$ is only consists of non-removable singularities; hence, $\dim(S_{\varPhi_x}) < 1$. However, if there is any point $y \in S_{\varPhi_x} \setminus \{0\}$, since $\varPhi_x$ is a homogenous map, the whole ray $\{ ry \; | \; r \in \mathbb{R}^+ \} \subset S_{\varPhi_x}$. This implies $\dim(S_{\varPhi_x}) \geq 1$, which is a contradiction. Therefore,
\begin{align*}
    S_{\varPhi_x} = \{0\} \subset \mathbb{R}^3.
\end{align*}
\end{proof}

In particular, the map $\varphi_x: S^2 \to \mathfrak{X}_x$ is smooth everywhere on $S^2$.

\begin{proof}[Proof of Theorem \ref{tri-hol-image}]
    Let $\varphi_{x}: S^2 \to \mathfrak{X}_{x}$ be the restriction of a tangent map of a Fueter section at a point $x$. Maps $f_{r_i}$ satisfy the Fueter equation at all the non-singular points, and therefore, $\varPhi_x$ is a Fueter map on $\mathbb{R}^3 \setminus \{0\}$. 

    Let $y \in S^2$. Choose the local frame $(\partial r, e_1, e_2)$ such that $\partial r$ is the unit radial vector, and  $(e_1, e_2)$ is an orthonormal frame on a neighborhood of $y$ in $S^2$. Then, we have
\begin{align*}
    I(\partial r) d \varPhi_x (\partial r) &+ 
    I(e_1) d \varPhi_x (e_1) + 
    I(e_2) d \varPhi_x (e_2) \\&=  
    I(e_1) d \varPhi_x (e_1) + 
    I(e_2) d \varPhi_x (e_2) = 0,
\end{align*}
since $d \varPhi_x (\partial r) = 0$.

By multiplying the equation by $-I(e_1)$, we get 
\begin{align*}  
    d \varPhi_x (e_1) - 
    I(\partial r) d \varPhi_x (e_2) = 0.
\end{align*}
The contractibility of the image of $\varphi_x$ follows from an argument in \cite[Footnote 32]{MR3941492}. In fact, there are two limits in this problem. After re-enumerating the subsequences, suppose, $f_{i} \to f$ and $f_{r_{j}} \to \varPhi_x$. One can combine these two sequences to construct a double sequence, 
\begin{align*}
    \tilde{f}_{i,j} := (f_{i})_{r_j} = f_i \circ \exp_x \circ \lambda_{r_j}.
\end{align*}
We have $\lim_{i \to \infty} \tilde{f}_{i,i} = \varPhi_x$ in $C^{\infty}_{\text{loc}}$. Therefore, $\varPhi_x$ can be understood as the limit of a sequence of smooth Fueter sections. 

For each $i$, let $S^2_i$ be a sufficiently small $2$-sphere in $\mathbb{R}^3$. The subset $\tilde{f}_{i,j} (S^2_i)$ is contractible since the $S^2_i$ is the boundary of a 3-dimensional ball, denoted by $B^3_i$, where $\tilde{f}_{i,j}$ is defined and smooth. By taking limit, same holds for $\varphi_x(S^2) = \varPhi_x(S^2)$, as explained in \cite{MR3941492}.
\end{proof}

\begin{proof}[Proof of Theorem \ref{non-constant-tangent}] This is similar to the case of energy minimizing maps \cite{MR664498, MR1399562}. We have
    \begin{align*}
        x \in M \setminus \text{sing}(f) \iff 
        \theta_f(x) = 0 \iff 
        \theta_{\varPhi_x}(0) = 0 \iff 
        d \varPhi_x = 0.
    \end{align*}
Therefore, $\varPhi_x$ is non-trivial when $x \in \text{sing}(f)$.
\end{proof}

\subsection{Analysis of the tri-holomorphic maps}\label{analysis-tri-holomorphic}

The previous subsection motivates the study of the tri-holomorphic maps to hyperk\"ahler manifolds. In this section, we show that tri-holomorphic maps are harmonic, and their images are minimal spheres. 

\begin {theorem} \label{def:hhh}
Let $f: S^2 \to X$ be a $I$-tri-holomorphic map. Then, $f$ is a harmonic map. 
\end {theorem}

\begin{proof}
Let $j$ be the standard almost complex structure on $S^2$. For a point $x \in S^2$ and $v_1$ and $v_2 \in T_x S^2$, we have
\begin{align*}
    \nabla df (v_1, j v_2) = \nabla_{v_1} df (j v_2) - df \nabla_{v_1} (j{v_2}).
\end{align*}
At the point $x \in S^2$, $d_xf$ is $I(x)$-linear, $d_xf (j v_2) = I(x) d_x f(v_2)$. Therefore, 
\begin{align*}
\nabla d_xf (v_1,jv_2) = I(x) \nabla d_xf (v_1,v_2).
\end{align*}
Hence, since $\nabla df$ is a symmetric tensor, we have
\begin{align*}
    \nabla df (j v_1,j v_1) = - \nabla df (v_1,v_2).
\end{align*}
Let $e_1$ and $e_2 = j e_1$ be a local orthonormal frame at $x \in S^2$. Let $\tau(f)$ denote the tension field of $f$ at $x$, 
\begin{align*}
\tau(f) = trace (\nabla df) &= \nabla df (e_1, e_1) + \nabla df (e_2, e_2) \\&= \nabla df (e_1, e_1) + \nabla df (je_1, je_1) = 0,
\end{align*}
and therefore, $f$ is harmonic.
\end{proof}

\begin{corollary}\label{tri-hol-min}
The image of an $I$-tri-holomorphic map $f:S^2 \to X$ is a minimal sphere.
\end{corollary}

\begin{proof}
At any point $x \in S^2$, the map $f: S^2 \to X$ is $I(x)$-holomorphic, and therefore, it is an angle-preserving conformal map. The image of a harmonic conformal map $f: S^2 \to X$ is a minimal surface in $X$.
\end{proof}

We finish this subsection with a simple observation that these conformal tri-holomorphic maps are not holomorphic with respect to any fixed holomorphic structure on $X$.

\begin{theorem}\label{noholomoprhic}
Let $f: S^2 \to X$ be an $I$-tri-holomorphic map. Furthermore, suppose $f: S^2 \to X$ is holomorphic with respect to a complex structure $I_0$ on $X$. Then, $f$ is a constant map.
\end{theorem}

\begin{proof}
At each point $x \in S^2$, we have 
\begin{align*}
    d_xf + I(x) \circ d_xf \circ j = 0, \quad \text{ and } \quad 
    d_xf + I_0 \circ d_xf \circ j = 0,
\end{align*}
and therefore,
\begin{align*}
    (I(x) - I_0) \circ df \circ j = 0,
\end{align*}
for all $x \in S^2$, and therefore, $df = 0$ almost everywhere on $S^2$; hence, $f$ is a constant map.
\end{proof}

Note that in Theorems \ref{def:hhh} and \ref{noholomoprhic}, and Corollary \ref{tri-hol-min}, one can replace $I$ with $-I$.

\subsection{Tri-holomorphic maps and monopoles}\label{tri-hol-mon}

In this subsection, we rule out the formation of the non-removable singularities for Fueter sections of charge 2 monopole bundles. 

\begin{theorem}[Theorem \ref{Main-theorem}]
Let $(M,g)$ be an oriented Riemannian 3-manifold. Let $\mathfrak{X}_2 \to M$ be the monopole bundle with charge 2. Let $K \subset \mathfrak{X}_2$ be a compact subset. Let $(f_i)_{i=1}^{\infty} \subset \Gamma(\mathfrak{X}_2)$ be a sequence of smooth Fueter sections such that $\text{Image}(f_i) \subset K$ for all $i$. Then, after passing to a subsequence, $(f_i)_{i=1}^{\infty}$ converges to a smooth Fueter section $f \in \Gamma (\mathfrak{X}_2)$ in $C^{\infty}(M)$.
\end{theorem}

The vanishing of the singular locus relies on two observations. Firstly, by the Theorem of Tsai and Wang \ref{Tsai-Wang-thm}, there is only one minimal sphere (in fact $\mathbb{RP}^2$) in the moduli space of centered monopoles with charge $2$. Secondly, following Theorems \ref{tri-hol-image} and \ref{non-constant-tangent}, this unique minimal sphere cannot be a source of non-removable singularity. 

\begin{proof}[Proof of Theorem \ref{Main-theorem}]
Following Theorem \ref{energy-bound}, there is a bound on the energy of $f_i$ for all $i$. Therefore, we can apply Walpuski's compactness theorem. There exists a subsequence $f_{i_j} \to f$ on $M \setminus S$ in $C^{\infty}_{\text{loc}}$, where $S = \Gamma \cup \text{sing}(f)$. 

Theorem \ref{no-bubbling-locus} shows $\Gamma = \emptyset$. Therefore, $S = \text{sing}(f)$. Let $x \in \text{sing}(f)$. By Theorem \ref{sing-fueter}, we have a tangent map at this point, $\varPhi_x : \mathbb{R}^3 \to (\mathfrak{X}_2)_x$, and by the restriction to the 2-sphere, a $-I$-tri-holomorphic map $\varphi_x : S^2 \to (\mathfrak{X}_2)_x$. Theorem \ref{non-constant-tangent} shows $\varphi_x$ is a non-trivial $-I$-tri-holomorphic map, and therefore, by Corollary \ref{tri-hol-min}, its image is a non-trivial compact minimal surface. 

Based on the theorem of Tsai and Wang, the image can only be the minimal surface of axisymmetric monopoles. However, Theorem \ref{tri-hol-image} shows that the image of $\varphi_x$ is contractible in the moduli space of charge 2 centered $SU(2)$ monopoles on $\mathbb{R}^3$, but this is not the case for the minimal surface of axisymmetric monopoles, since it is a calibrated submanifold. In fact, this $\mathbb{RP}^2$ is the generator of $\pi_2(\text{Mon}_2^{\circ}(\mathbb{R}^3))$, and therefore, $S = \emptyset$.
\end{proof}

\begin{remark}
    Let $\varPhi: \mathbb{R}^3 \setminus \{0\} \to \text{Mon}_2^{\circ}(\mathbb{R}^3)$ be the map which sends any $x \in \mathbb{R}^3 \setminus \{0\}$ to the unique centered charge 2 monopole on $\mathbb{R}^3$ which is invariant under rotation around the axis generated by $x$. $\varPhi$ is a Fueter map, and $\text{Image}(\varPhi) = \Sigma_2$. The map $\varPhi$ has a non-removable singularity at the origin $0 \in \mathbb{R}^3$. However, $\varPhi$ cannot appear as a tangent map of a limiting Fueter section $f$.
\end{remark}

In the following subsection, with a small digression, we give a description of tri-holomorphic spheres in terms of sections of twistor spaces.

\subsection{A twistorial description of tri-holomorphic maps}\label{twis}

A standard method to study hyperk\"ahler manifolds is to use the twistorial description, introduced by Penrose in \cite{MR216828}, and more relevant to our setup by Hitchin in \cite{MR877637}.

In this section, we give a twistorial description of the tri-holomorphic maps and show that the $-I$-tri-holomorphic maps introduced by the non-removable singularities of Fueter sections can be understood as sections of the twistor spaces over $\mathbb{CP}^1$, where these sections are holomorphic with respect to a certain non-integrable almost complex structure.

As we mentioned earlier, if $(X, g_X, I, J, K)$ is a hyperk\"ahler manifold, then any $(aI + bJ + cK)$ with $a^2 + b^2 + c^2 = 1$ is also a covariantly constant complex structure on $X$. The twistor space $Z$ associated to a $4k$-dimensional hyperk\"ahler manifold $X$, topologically, is defined to be the product manifold $X \times S^2$ with the almost complex structure 
\begin{align*}
   \mathcal{J}_1 = (aI + bJ + cK, j),
\end{align*}
at the point $(x,a,b,c) \in X \times S^2$, where $j$ is the standard complex structure on $S^2$. We call this complex structure the standard complex structure of the twistor space. This complex structure is integrable.

\begin{theorem} [Atiyah-Hitchin-Singer \cite{MR877637, MR506229}]
$(Z,\mathcal{J}_1)$ is a complex manifold. 
\end{theorem}

There is also a second almost complex structure defined on $Z$, introduced by Eells and Salamon \cite{MR848842}. This almost complex structure on the twistor space is defined by 
\begin{align*}
   \mathcal{J}_2 = (-aI - bJ - cK, j), 
\end{align*}
at each point $(x,a,b,c) \in X \times S^2$. 

This almost complex structure is called the Eells-Salamon almost complex structure. These two complex structures are quite different. For instance, we have the following theorem.

\begin{theorem}[Eells-Salamon \cite{MR848842}]
The almost complex structure $\mathcal{J}_2$ on $Z$ is non-integrable.
\end{theorem}

As observed by Eells and Salamon, the twistor space equipped with this non-integrable complex structure is quite suitable for studying the harmonic maps into the hyperk\"ahler manifolds. It turns out that the Eells-Salamon twistor space $(Z, \mathcal{J}_2)$ gives an amenable description of the $-I$-tri-holomorphic spheres in $X$.

\begin{theorem}\label{J2holomorphic}
Let $(X, g_X, I, J, K)$ be a $4k$-dimensional hyperk\"ahler manifold. Let $Z = S^2 \times X$ be the twistor space. For any map $f : S^2 \rightarrow X$, let $\tilde{f}: S^2 \rightarrow Z$ be the map defined by
\begin{align*}
    \tilde{f} (x) := (x , f(x)).
\end{align*}
Then, the map $f: S^2 \to X$ is a $-I$-tri-holomorphic map if and only if $\tilde{f}$ is a $\mathcal{J}_2$-holomorphic section of the twistor space $Z \to S^2$.
\end{theorem}

\begin{proof}
The theorem follows from the following,
\begin{align*}
    2 \overline{\partial}_{\mathcal{J}_2} \tilde{f}(x) = 
    d_x\tilde{f} + 
    \mathcal{J}_2 \circ d_x \tilde{f} \circ j = (0, df - I(x) \circ d_x f \circ j) = (0, 2 \overline{\partial}_{-I(x)} f).
\end{align*}
\end{proof}

In fact, the twistor space can be defined for any oriented Riemannian $4$-manifold $X$. The Grassmann bundle of 2-planes
\begin{align*}
Z := \widetilde{Gr}_2(TX) \to X,   
\end{align*}
is a vector bundle whose fiber at the point $x \in X$ is the space of real oriented two dimensional subspaces in $T_x X$, denoted by $\widetilde{Gr}_2(T_xX)$, which is the double cover of $Gr_2(T_xX)$.  

Let
\begin{align*}
    S_{\pm} = S(\Lambda^{\pm} T^*X),
\end{align*}
be the unit 2-sphere bundles of self-dual and anti-self-dual 2-forms on $X$. We have the following bundle isomorphism,
\begin{align*}
    \widetilde{Gr}_2(TX) \cong S_+ \times S_-.
\end{align*}
The bundle $S_-$ generalizes the twistor bundle $Z$ we constructed above for hyperk\"ahler 4-manifolds to general oriented Riemannian 4-manifolds. The almost complex structure $\mathcal{J}_2$ can also be defined on $S_-$ \cite{MR848842}. 

We have projection maps
\begin{align*}
    \pi_{\pm} : \widetilde{Gr}_2(TX) \rightarrow S_{\pm}.
\end{align*}
For any map $f : S^2 \to X$, let $\widetilde{f}_{\pm} : S^2 \to S_{\pm}$ be the maps defined by
\begin{align*}
 \widetilde{f}_{\pm} = \pi_{\pm} \circ \widetilde{f}.    
\end{align*}

\begin{theorem}[Eells-Salamon \cite{MR848842}] \label{ES}
Let $(X, g_X)$ be an oriented Riemannian 4-manifold. The correspondences
\begin{align*}
    f \to \widetilde{f}_{+}, \quad  \text{ and } \quad f \to \widetilde{f}_{-},
\end{align*}
are bijective correspondences between non-constant conformal harmonic maps $f:S^2 \rightarrow X$ and non-trivial $\mathcal{J}_2$-holomorphic maps $\widetilde{f}_{\pm}: S^2 \rightarrow S_{\pm}$.
\end{theorem}

This theorem gives another argument for Corollary \ref{tri-hol-min} that the image of a $-I$-tri-holomorphic map $f: S^2 \to X$ is a minimal sphere.

\begin{proof}[Second Proof for Corollary \ref{tri-hol-min} ]
Let $f:S^2 \rightarrow X$ be a  $-I$-tri-holomorphic map. The corresponding section of $Z \to S^2$, denoted by $\tilde{f}$, is a $\mathcal{J}_2$-holomorphic section of the bundle, and therefore, $\tilde{f}_-$ is a $\mathcal{J}_2$-holomorphic section of $S_- \to S^2$. By the Theorem \ref{ES}, $f$ is a conformal harmonic map, and therefore, its image in $X$ is a minimal surface in $X$. 
\end{proof}

\section{Appendix: 4-dimensional case}

In this appendix, we study 4-dimensional Fueter sections and prove a compactness theorem for these Fueter sections in the case of monopole bundles with charge 2.

We give four items of motivation for this study.

\begin{itemize}
    \item The cylindrical 4-dimensional Fueter sections appear in the definition of differentials of the monopole Fueter Floer homology of 3-manifolds. 

    \item The space of Fueter sections of monopole bundles over 4-manifolds can potentially be used to define a smooth 4-manifold invariant. This invariant can be thought of as a 4-dimensional version of Rozansky-Witten invariants.

    \item Fueter sections of monopole bundles over coassiative 4-manifolds in $G_2$-manifolds appear in the study of $G_2$-monopoles. One expects a $G_2$-version of the Donaldson-Segal conjecture, i.e., a correspondence between the moduli spaces of $G_2$-monopoles and coassociative submanifolds weighted with the Fueter sections of monopole bundles above them.

    \item The analysis is expected to be relevant to a 4-dimensional version of Taubes' generalized Seiberg-Witten equations.
    
\end{itemize}

We start with the definition of the 4-dimensional Fueter sections.

\begin{definition} Let $(N,h)$ be an orientable Riemannian 4-manifold. Let $\mathfrak{X} \to N$ be a hyperk\"ahler bundle with a connection $\nabla$, with a fixed isometric bundle identification $I: S (\Lambda^+ T^*N) \to \mathfrak{b}(\mathfrak{X})$, and $\iota : \Lambda^+ T^*N \to \mathfrak{so}(T^*N)$, where $S (\Lambda^+ T^*N)$ is the unit sphere bundle of self-dual 2-forms on $N$ and $\mathfrak{b}(\mathfrak{X})$ is the 2-sphere bundle of K\"ahler structures of the fibers of $\mathfrak{X}$. A section $f \in \Gamma(\mathfrak{X})$ is called a 4-dimensional Fueter section if
\begin{align*}
    \mathfrak{F}_4(f) := \nabla f - \sum_{i=1}^3 I(\Omega_i) \circ \nabla f \circ \iota (\Omega_i) = 0,
\end{align*}
where $(\Omega^1, \Omega^2, \Omega^3)$ is a local orthonormal basis of $\Lambda^+ T^*N$. We call $\mathfrak{F}_4$ the 4-dimensional Fueter operator.
\end{definition}

The monopole bundle with charge $k$ over 4-manifolds is defined similar to the 3-dimensional case, where the $SO(3)$-frame bundle in the definition of the monopole bundle is replaced by the frame bundle of $\Lambda^ +T^*N$. The monopole bundle with charge $k$ is defined by $Fr_{SO(3)} \times_{SO(3)} \mathfrak{X}_k \to N$. Similar to the 3-dimensional case, there is a natural connection induced by the Levi-Civita connection, and isometric identifications $I$ and $\iota$. 

\begin{theorem}\label{four-dim-main}
Let $(N,h)$ be an oriented Riemannian 4-manifold. Let $\mathfrak{X}_2 \to N$ be the monopole bundle with charge 2. Let $K \subset \mathfrak{X}_2$ be a compact subset. Let $(f_i)_{i=1}^{\infty} \subset \Gamma(\mathfrak{X}_2)$ be a sequence of smooth 4-dimensional Fueter sections such that $\text{Image}(f_i) \subset K$ for all $i$. Then, after passing to a subsequence, $(f_i)_{i=1}^{\infty}$ converges to a smooth Fueter section $f \in \Gamma (\mathfrak{X}_2)$ in $C^{\infty}(N)$.
\end{theorem}

\begin{proof}
    The analysis in the 4-dimensional case is similar to the 3-dimensional case, and therefore, we will be brief with our argument. 

    \begin{itemize}
        \item Walpuski's compactness theorem for 4-dimensional Fueter sections \cite[Theorem B.6.]{MR3718486} shows that if $\frac{1}{2} \int_N | \nabla f_i |^2 \leq c_e$, for a constant $c_e$, then, after passing to a subsequence, $f_i \to f$, for a Fueter section $f$ on the complement of a singular locus $S = \Gamma \cup \text{sing}(f)$ in $C^{\infty}_{\text{loc}}$, with $\mathcal{H}^2(S) < \infty$. The bubbling locus $\Gamma$ is  $\mathcal{H}^2$-rectifiable and is the locus where the energy loss happens. For any smooth point $x \in \Gamma$, there is a non-trivial holomorphic sphere in the fiber above $x$. The singular set is defined by $\text{sing}(f) =  \{ x \in N \; | \;  \limsup_{r \downarrow 0} \frac{1}{r^2} \int_{B_r (x)} |\nabla f|^2 > 0 \}$. Furtheremore, $\mathcal{H}^2(\text{sing}(f)) = 0$.

    \item Energy Bound: There is a bound on the energy of Fueter sections when their images are inside a compact subset of $\mathfrak{X}_2$. To see this, let $(\Omega^1, \Omega^2, \Omega^3)$ be a local orthonormal frame of $\Lambda^+ T^*N$ on a neighbourhood of $x$. Let
\begin{align*}
    \Lambda_4 := \sum_{i=1}^3 \Omega^i \wedge f^*(\omega_{I(\Omega^i)}) \in \Omega^4(N),
\end{align*}
where $\omega_{I(\Omega^i)}$ is the symplectic structure associated to the complex structure $I(\Omega^i)$. This 3-form is independent of the chosen local orthonormal frame, and therefore, it is defined globally on $N$. 

For any section $f \in \Gamma(\mathfrak{X}_2)$, we have 
\begin{align*}
    \|\nabla f\|_{L^2}^2 = \frac{1}{2}\| \mathfrak{F}_4(f)\|_{L^2}^2 - 2\int_{N} \Lambda_4.
\end{align*}
This identity follows from a pointwise computation similar to \cite[Proposition 2.2]{MR1847314}. Therefore, if $f$ is a 4-dimensional Fueter section, we have $\mathcal{E}(f) =  - \int_{N} \Lambda_4$. Let $\omega_{I(\Omega^i)} = d\alpha_{I(\Omega^i)}$. By integration by parts, 
\begin{align*}
    \|\nabla f\|_{L^2(B_{\epsilon}(x))}^2 &=
     \int_{B_{\epsilon}(x)} \sum_{i=1}^3 d\Omega^i \wedge f^*(\alpha_{I(\Omega^i)}) -  \int_{\partial B_{\epsilon}(x)} \sum_{i=1}^3 \Omega^i \wedge f^*(\alpha_{I(\Omega^i)}).
\end{align*}
Similar to the 3-dimensional case, the boundary terms cancel. The terms $\|d \Omega^i\|_{L^2}$ only depend on the geometry of $(N,h)$, and therefore, by H\"older's inequality, we have
\begin{align*}
    \|\nabla f\|_{L^2} \leq c r(K).
\end{align*}

\item The bubbling locus $\Gamma$ vanishes: This follows from the fact that the moduli space of centered monopoles on $\mathbb{R}^3$ with any charge $k$ does not contain any non-trivial holomorphic sphere, and therefore, $\mathcal{H}^2(\Gamma) = 0$. Hence, $\Gamma := \text{supp}(\Theta \mathcal{H}^2_{|_{S}}) = \emptyset$. 

\item The locus of non-removable singularities vanishes: let $x \in \text{sing}(f)$. By blowing up $f$ at $x$ one get a non-trivial homogenous tangent map
\begin{align*}
\varPhi_x: \mathbb{R}^4 \setminus \{0\} \to (\mathfrak{X}_2)_x.    
\end{align*}
The analysis uses an almost monotonicity formula for 
\begin{align*}
\frac{1}{r^2}\int_{B_r(x)} |\nabla f|^2.    
\end{align*}
By restriction of $\varPhi_x$ to the 3-sphere $S^3 \subset \mathbb{R}^4$, we get a Fueter map 
\begin{align*}
\varphi_x: S^3 = SU(2) \to (\mathfrak{X}_2)_x,    
\end{align*}
with respect to a lef-invariant frame $(v_1,v_2,v_2)$ which at $1 \in SU(2)$ is $(i,j,k)$. The image of $\varphi_x$ is a non-trivial compact contractible minimal 3-dimensional submanifold in the fiber above $x$. Following, Tsai-Wang's theorem, there is no compact minimal 3-dimensional submanifold in the moduli space of centered monopoles with charge $2$, and therefore, $\text{sing}(f) = \emptyset$. 

\end{itemize}

\end{proof}

We continue with a short discussion of the trajectories of the Floer theory. This is similar to the case studied in \cite{MR2529942}.

Let $f_{+}$ and $f_{-}$ be Fueter sections of $\mathfrak{X}_2$ over a 3-manifold $(M,g)$. Let $(f_i)_{i=1}^{\infty}$ be a sequence 4-dimensional Fueter sections over $N = M \times \mathbb{R}_t$ which appear as the trajectories between $f_{+}$ and $f_{-}$. 
\begin{align*}
    \lim_{t \to  \infty} f_{i |_{N \times \{t\}}} = f_{+}, \quad \text{ and } \quad
    \lim_{t \to -\infty} f_{i |_{N \times \{t\}}} = f_{-}.
\end{align*}
Moreover, suppose $\text{Image}(f_i) \subset K \subset \mathfrak{X}_2$, for a compact subset $K \subset \mathfrak{X}_2$ and all $i$. Then, $\mathcal{E}(f_i) \leq c_e$, for a constant $c_e$. By repeating a similar analysis as in the closed case on $M \times [a,b]$, for any $- \infty < a < b < \infty$, and ruling out the formation of bubbles and non-removable singularities, after passing to a subsequence, we have $f_i \to f$ in $C^{\infty}$, such that
\begin{align*}
    \lim_{t \to  \infty} f_{|_{N \times \{t\}}} = f_{+}, \quad \text{ and } \quad
    \lim_{t \to -\infty} f_{|_{N \times \{t\}}} = f_{-}.
\end{align*}

We finish the appendix with a rather speculative remark. 

\begin{remark}
Assuming the $k$-th monopole Fueter Floer homology can be defined for closed oriented Riemannian 3-manifolds, one can hope to define a new Calab-Yau Floer theory for Calabi-Yau 3-folds and categorify the Donaldson-Segal counts of weighted special Lagrangians. 

Let $Z$ be a Calabi-Yau 3-fold. For each class $\kappa \in H_3(Z)$, let $HM_{\kappa}(Z)$ be the monopole Calabi-Yau Floer homology of $Z$ associated to $\kappa$, defined by
\begin{align*}
    HM_{\kappa}(Z) := \bigoplus_{\kappa = \sum_i k_i [L_i]} \bigotimes_i HM_{k_i}(L_i),
\end{align*}
where $L_i$ are the special Lagrangians in $Z$ and $HM_{k_i}(L_i)$ is the $k_i$-th monopole Fueter Floer homology of $L_i$. 
\end{remark}

\hspace{10pt}

\printbibliography

\vspace{10pt}

\noindent
\author{Department of Mathematics, Duke University, Durham, NC, 27708
} \\ E-mail address: \href{ mailto:Saman.HabibiEsfahani@msri.org}{saman.habibiesfahani@duke.edu}

\end{document}